\documentclass[12pt]{amsart}

\usepackage{ucs}

\usepackage{amssymb}
\usepackage{amsthm}
\usepackage{amsmath}
\usepackage{latexsym}
\usepackage[cp1251]{inputenc}
\usepackage{graphicx}
\usepackage{wrapfig}
\usepackage{caption}
\usepackage{subcaption}
\usepackage{indentfirst}
\usepackage[left=2.6cm,right=2.6cm,top=2.6cm,bottom=2.6cm,bindingoffset=0cm]{geometry}
\usepackage{enumerate}
\usepackage{makecell}

\DeclareMathOperator{\aut}{Aut}

\DeclareMathOperator{\diag}{Diag}

\DeclareMathOperator{\id}{id}

\DeclareMathOperator{\rk}{rk}
\DeclareMathOperator{\cork}{cork}

\DeclareMathOperator{\Span}{Span}

\DeclareMathOperator{\sym}{Sym}

\def\tm#1{\item[{\rm (#1)}]}

\makeatletter 
\def\@seccntformat#1{\csname the#1\endcsname. } 
\def\@biblabel#1{#1.}

\makeatother

\title{On uniform Higmanian association schemes}

\author{Grigory Ryabov}
\address{Sobolev Institute of Mathematics, Novosibirsk, Russia}
\email{gric2ryabov@gmail.com}

\thanks{The author is supported by the state contract of the Sobolev Institute of Mathematics (project number FWNF-2026-0011)}

\date{}

\newtheorem{prop}{Proposition}[section]

\newtheorem{lemm}[prop]{Lemma}
\newtheorem{theo}[prop]{Theorem}

\newtheorem{corl}[prop]{Corollary}

\theoremstyle{definition}

\newtheorem*{rem}{Remark}

\begin{document}

\maketitle

\begin{abstract}
An imprimitive symmetric indecomposable association scheme of rank~$5$ is said to be \emph{Higmanian}. In the present paper, we prove a necessary and sufficient condition for a Higmanian association scheme with two nontrivial parabolics to be uniform. We also provide examples of uniform Higmanian Cayley schemes.
\end{abstract}

\section{Introduction}
In the present paper, we study the property of (association) schemes to be \emph{uniform}. The latter means informally that the scheme is imprimitive and its intersection numbers are divided uniformly over the classes of a nontrivial parabolic the quotient by which is trivial (see Subsection~2.3 for the exact definition). The concept of uniformity was introduced by Higman for schemes of rank~$5$ in the paper~\cite{Hig} and for arbitrary schemes in the unpublished manuscript in~\cite{Hig0}. Higman established the connection of some uniform schemes of rank~$5$ with linked systems of strongly regular designs and finite geometries and gave examples of uniform schemes constructed from the action of several finite simple groups (so-called \emph{triality} schemes).

Throughout the paper, we follow in general the terminology from~\cite{DMM}, where the authors studied uniform schemes motivated by the search for Q-polynomial schemes (see~\cite{BI} for the definition). They provide several criteria for a scheme to be uniform. In particular, the authors proved that a scheme is uniform if and only if it is \emph{dismantlable}~\cite{MMW}, i.e. a restriction of the scheme on a union of any classes of a nontrivial parabolic is a scheme. 

We are interested in uniformity of \emph{Higmanian} schemes, i.e. imprimitive symmetric schemes of rank~$5$ which are not wreath products. Studying such schemes was initiated by Higman in~\cite{Hig}, who was motivated by a connection of these schemes with linked systems of strongly regular designs, finite geometries, and actions of finite simple groups. The term ``Higmanian scheme'' was introduced in~\cite{KMZ}, where several computational results on such schemes were obtained.

More precisely, we deal with the special case of a Higmanian scheme, namely, with a Higmanian scheme with two nontrivial parabolics (schemes from the intersection of Classes I and II in the sense of~\cite{Hig}). In this paper, we consider only such Higmanian schemes and the term ``Higmanian scheme'' means a Higmanian scheme with two nontrivial parabolics from now throughout the paper.

One of the reasons why uniform Higmanian schemes seem interesting is that they are closely related to linked systems of (group) divisible designs. In~\cite[Theorem~4.3]{KhS}, it was proved that from every uniform Higmanian scheme, it is possible to construct a linked system of divisible designs possessing special properties. On the other hand, from every linked system of divisible designs with these properties, it is possible to construct a uniform Higmanian scheme~\cite[Theorem~4.1]{KhS}. In this paper, we provide examples of uniform Higmanian schemes which are Cayley schemes arising from linked systems of relative difference sets (Section~$4$).

All parameters of a Higmanian scheme can be expressed via five of them (see~\cite[Section~10]{Hig}). In the current paper, it is convenient to use the following five parameters of a Higmanian scheme $\mathcal{X}$ with two nontrivial parabolics $E\subsetneq F$ and two basic relations $S$ and $T$ outside $F$ such that the valency $n_T$ of $T$ does not exceed the valency $n_S$ of $S$: the index~$f$ of $F$ in the whole scheme, the index~$m$ of $E$ in $F$, the valency~$n$ of $E$, the number~$k$ of neighbors of a given point in $S$ that lie in the same class of $F$, and the intersection number $t=p_{TS}^{T}$ (see Subsection~2.2 for the exact definitions). Throughout this paper, the numbers~$(f,m,n,k,t)$ are called the \emph{parameters} of $\mathcal{X}$. By the definitions of the parameters, we have $k\leq mn$. Since the parabolics $E\subsetneq F$ are nontrivial, we have $f,m,n\geq 2$. A criterion of uniformity for a Higmanian scheme can be expressed as an equality depending on the above five parameters.

\begin{theo}\label{main}
A Higmanian scheme with parameters~$(f,m,n,k,t)$, $f,m,n\geq 2$, is uniform if and only if 
\begin{equation}\label{uniform}
t=\frac{k(f-2)}{mn}\left(mn-k\pm \sqrt{\frac{k(mn-k)}{m(n-1)}}\right).
\end{equation}
\end{theo}

We prove Theorem~\ref{main} in Section~$3$. To do this, we use a criterion of uniformity for schemes~\cite[Theorem~4.16]{DMM} stating that a scheme is uniform if and only if it is \emph{Q-Higmanian}, i.e. there is an ordering of primitive idempotents of this scheme, satisfying special conditions (see Subsection~2.3 for the exact definitions and formulations). Section~$2$ contains a necessary background of schemes.

The author is very grateful to Prof. M. Muzychuk for the discussions on the subject matters.

\section{Preliminaries}

In what follows, we freely use basic facts concerned with (association) schemes, $S$-rings, and relative and divisible difference sets (RDSs and DDSs, resp.) for a background of which we refer the reader to~\cite{BI,CP},~\cite{MP,Ry},~and~\cite{MR,Pott}, respectively. We also recall notation and statements necessary for the explanation of the material of this paper.

\subsection{Schemes}

Let $\mathcal{X}$ be an (association) scheme on a finite set $\Omega$ of cardinality $v$ with the set of basic relations $\mathcal{R}=\mathcal{R}(\mathcal{X})$. Given $R\in \mathcal{R}$, the relation inverse to $R$ is denoted by $R^*$. We assume that the set $\mathcal{R}$ is ordered as 
$$\mathcal{R}=\{R_0,R_1,\ldots,R_d\},$$
 where $R_0$ is the diagonal of $\Omega\times \Omega$ and $d+1=\rk(\mathcal{X})$ is the rank of $\mathcal{X}$. This ordering induces the ordering $A_0,\ldots,A_d$ of the set of adjacency matrices, where $A_i$ is the adjacency matrix of $R_i$, $i\in\{0,\ldots,d\}$. We also assume that there is an ordering 
$$\mathcal{E}=\{E_0,E_1,\ldots,E_d\}$$ 
of the set $\mathcal{E}$ of primitive idempotents of $\mathcal{X}$, where $vE_0$ is the all-identities matrix. The first eigenmatrix of $\mathcal{X}$ with respect to these orderings of the basic relations and primitive idempotents is denoted by $P$. The elements of $P$ are denoted by $P_{ij}$. Put also 
$$I_0=\{0,1,\ldots,d\}.$$

Given $R=R_i,S=R_j,T=R_k\in \mathcal{R}$ ($E_i,E_j,E_k\in \mathcal{E}$, resp.), the corresponding intersection number (Krein parameter, resp.) of $\mathcal{X}$ is denoted by $p_{RS}^T=p_{ij}^k$ (by $q_{ij}^k$, resp.). The valency of a basic relation $R=R_i\in \mathcal{R}$ and a multiplicity of a primitive idempotent $E_j$ are denoted by $n_R=n_i$ and $m_j$, respectively. If $\mathcal{X}$ is symmetric, i.e. $R=R^*$ for every $R\in \mathcal{R}$, then all entries of the matrix $P$ are real. Therefore
\begin{equation}\label{krein} 
q_{ij}^k=\frac{m_im_j}{v}\sum \limits_{l\in I_0} \frac{P_{il}P_{jl}P_{kl}}{n_l^2}
\end{equation}
and
\begin{equation}\label{mult} 
m_j=v\left(\sum \limits_{i\in I_0} \frac{P_{ji}^2}{n_i}\right)^{-1}
\end{equation}
for all $i,j,k\in I_0$. Eqs.~\eqref{krein} and \eqref{mult} are special cases of~\cite[Theorem~3.6(i)]{BI} and~\cite[Theorem~4.1(i)]{BI}, respectively. 

If $E$ is a parabolic of $\mathcal{X}$, then the set of all classes of $E$ is denoted by $\Omega/E$. All classes of $E$ have the same size equal to the valency of $E$ and denoted by $n_E$. Put $\rk(E)=\rk(\mathcal{X}_{\Delta})$, where $\Delta\in \Omega/E$ and $\mathcal{X}_{\Delta}$ is the restriction of $\mathcal{X}$ on $\Delta$, and $\cork(E)=\rk(\mathcal{X}_{\Omega/E})$, where $\mathcal{X}_{\Omega/E}$ is the quotient of $\mathcal{X}$ modulo $E$. It is easy to verify that $\rk(E)$ does not depend on $\Delta\in \Omega/E$.

\subsection{Higmanian schemes}
Recall that a \emph{Higmanian scheme} is defined to be an imprimitive symmetric scheme of rank~$5$ which is indecomposable, i.e. not a wreath product of schemes. Throughout the paper, by a Higmanian scheme, we mean only Higmanian schemes $\mathcal{X}$ with two proper nontrivial parabolics $E\subsetneq F$ (schemes from the intersection of Classes I and~II in the sense of~\cite{Hig}). In this case, $\rk(E)=\cork(F)=2$ and $\rk(F)=\cork(E)=3$. In particular, the quotient scheme $\mathcal{X}_{\Omega/E}$ is the wreath product of two trivial schemes.

Put
$$f=\frac{|\Omega|}{n_F},~m=\frac{n_F}{n_E},~n=n_E.$$
As $E\subsetneq F$ are proper and nontrivial, we have $m,n,f\geq 2$. There are exactly two basic relations of $\mathcal{X}$, say $S$ and $T$, lying outside $F$. Without loss of generality, we assume that $n_T\leq n_S$. Due to~\cite[Exercise~2.7.11(1)]{CP} and the condition $\cork(F)=2$, the number 
$$k=|\alpha S\cap \Delta|,~\Delta\in \Omega/F,$$
does not depend on $\alpha\in \Omega$ and $\Delta\in \Omega/F$. Clearly, $k\leq mn$. With the above notation,
$$n_S=k(f-1)~\text{and}~n_T=(mn-k)(f-1).$$
Since $n_T\leq n_S$, we have $mn-k\leq k$. Finally, put
$$t=p_{TS}^{T}.$$ 
It can be verified that all intersection numbers of $\mathcal{X}$ can be expressed via $m$, $n$, $f$, $k$, and $t$ (see~\cite[Section~10]{Hig}). We say that the numbers $(f,m,n,k,t)$ are \emph{parameters} of $\mathcal{X}$.

\subsection{Uniform and Q-Higmanian schemes}

Throughout this subsection, $\mathcal{X}$ is assumed to be commutative, i.e. $p_{RS}^T=p_{SR}^T$ for all $R,S,T\in \mathcal{R}$. This always holds if $\mathcal{X}$ is symmetric. If $E$ is a parabolic of $\mathcal{X}$ and $\Delta,\Gamma\in \Omega/E$, then put
$$I(\Delta,\Gamma)=\{i\in I_0:~R_i\cap (\Delta\times\Gamma)\neq \varnothing\}.$$ 
If $A$ is a $(v\times v)$-matrix whose elements are indexed by the elements of $\Omega$ and $\Delta,\Gamma\subseteq \Omega$, then $A^{\Delta\Gamma}$ is defined to be the $(v\times v)$-matrix such that
$$(A^{\Delta\Gamma})_{xy}=
\begin{cases}
A_{xy},~\text{if}~x\in \Delta,~y\in \Gamma,\\
0, \text{otherwise}.
\end{cases}
$$

The scheme $\mathcal{X}$ is said to be \emph{uniform}~\cite[Definition~3.2]{DMM} if $\mathcal{X}$ has a nontrivial parabolic $E$ satisfying the following conditions:
\begin{enumerate} 

\tm{1} $\cork(E)=2$; 

\tm{2} for all $\Delta,\Gamma,\Lambda\in \Omega/E$, $i\in I(\Delta,\Gamma)$, and $j\in I(\Gamma,\Lambda)$, there are nonnegative integers $a_{ij}^k$, $k\in I_0$, such that
$$A_i^{\Delta\Gamma}A_j^{\Gamma\Lambda}=\sum \limits_{k\in I_0} a_{ij}^k A_k^{\Delta\Lambda}.$$
\end{enumerate}

The scheme $\mathcal{X}$ is imprimitive if and only if there exists $I\subseteq I_0$ such that $\langle E_i:~i\in I\rangle$ is a proper nontrivial $\circ$-subalgebra of the adjacency algebra $\mathcal{M}=\mathcal{M}(\mathcal{X})$ of $\mathcal{X}$. The set of all such $I$ is denoted by $\mathcal{I}$. With each $I\in \mathcal{I}$, one can associate a unique parabolic $E=E(I)$ of $\mathcal{X}$.

Given $I\in \mathcal{I}$, let us define a binary relation $\sim_I$ on $I_0$ by
$$i\sim_I j~\text{if and only if}~q_{ij}^{k}\neq 0~\text{for some}~k\in I.$$
Due to~\cite[Lemma~2.1]{DMM}, the relation~$\sim_I$ is an equivalence relation on $I_0$.

The scheme $\mathcal{X}$ is said to be \emph{Q-Higmanian} (see~\cite[Definition~4.6]{DMM}) if there exist an ordering of $I_0$ and $l\in\{1,\ldots,\frac{d}{2}\}$ such that:
\begin{enumerate}
\tm{1} $I=\{0,d\}\in \mathcal{I}$;

\tm{2} the equivalence classes of $\sim_I$ are $\{i,d-i\}$, $i\in\{0,\ldots,l-1\}$, and $\{i\}$, $i\in\{l,\ldots,d-l\}$;

\tm{3} $m_{d-i}=(f-1)m_i$ for $i\in\{0,\ldots,l-1\}$, where $f=|\Omega/E(I)|$.
\end{enumerate}

\noindent Clearly, a Q-Higmanian scheme is imprimitive. We use the term ``Q-Higmanian scheme'' in this paper instead of ``Q-Higman scheme'' used in~\cite{DMM} following in general the terminology from the paper~\cite{KMZ}, where the term ``Higmanian scheme'' was introduced.

\begin{lemm}\cite[Theorem~4.16]{DMM}\label{criterion}
A scheme is uniform if and only if it is Q-Higmanian.
\end{lemm}

\subsection{Cayley schemes and $S$-rings}

Let $G$ be a finite group and $\mathbb{Z}G$ the group ring over the integers. The identity element of $G$ is denoted by $e$. Given $X\subseteq G$, we set $$\underline{X}=\sum \limits_{x\in X} {x}\in\mathbb{Z}G,~X^{(-1)}=\{x^{-1}:x\in X\},~X^\#=X\setminus\{e\}.$$
It is easy to verify that
\begin{equation}\label{easy}
\underline{X}\cdot \underline{H}=\underline{H}\cdot \underline{X}=|X|\underline{H}
\end{equation}
for all $X\subseteq G$ and $H\leq G$ such that $X\subseteq H$.

Due to~\cite[Theorem~2.4.16]{CP}, there is a one-to-one correspondence 
$$\mathcal{A}\mapsto \mathcal{X}=\mathcal{X}(\mathcal{A})$$ 
between the $S$-rings and Cayley schemes over $G$. This correspondence induces the bijection $\rho$ between the set of all unions of basic sets of $\mathcal{A}$ and the set of all unions of basic relations of $\mathcal{X}$. Moreover, $\rho$ maps the basic sets of $\mathcal{A}$ to the basic relations of $\mathcal{X}$ and the $\mathcal{A}$-subgroups to the parabolics of $\mathcal{X}$. The set of all basic sets of $\mathcal{A}$ is denoted by $\mathcal{S}(\mathcal{A})$. Given $X,Y,Z\in \mathcal{S}(\mathcal{A})$, we have $\rho(X^{(-1)})=\rho(X)^*$ and $p_{XY}^Z=p_{\rho(X)\rho(Y)}^{\rho(Z)}$, where $p_{XY}^{Z}$ is the structure constant of $\mathcal{A}$ corresponding to $X,Y,Z$. Due to~\cite[Eq.~2.1.14]{CP}, we have 
\begin{equation}\label{triangle0}
|Z|p^{Z^{(-1)}}_{XY}=|X|p^{X^{(-1)}}_{YZ}=|Y|p^{Y^{(-1)}}_{ZX}
\end{equation}
for all $X,Y,Z\in \mathcal{S}(\mathcal{A})$.

Two $S$-rings are said to be \emph{Cayley isomorphic} if there exists an isomorphism between the underlying groups that maps the basic sets of one $S$-ring to the basic sets of the other one. In this case, the corresponding Cayley schemes are isomorphic. 

A symmetric $S$-ring $\mathcal{A}$ of rank~$5$ is said to be \emph{Higmanian} if $G$ has proper nontrivial $\mathcal{A}$-subgroups $L$ and $U$ such that $L<U$, $\rk(\mathcal{A}_L)=\rk(\mathcal{A}_{G/U})=2$, and $\rk(\mathcal{A}_{G/L})=\rk(\mathcal{A}_U)=3$. It follows from the definitions that $\mathcal{A}$ is Higmanian if and only if $\mathcal{X}(\mathcal{A})$ so is. The \emph{parameters} of a Higmanian $S$-ring $\mathcal{A}$ are $f=|G:U|$, $m=|U:L|$, $n=|L|$, $k=|X\cap Ug|$ not depending on $g\in G$, and $t=p_{YX}^Y$, where $X,Y\in \mathcal{S}(\mathcal{A})$ outside $U$ are such that $|Y|\leq |X|$. Clearly, the parameters of the Higmanian $S$-ring $\mathcal{A}$ and the parameters of the Higmanian scheme $\mathcal{X}(\mathcal{A})$ coincide. We say that the $S$-ring $\mathcal{A}$ is \emph{uniform} (\emph{Q-Higmanian}, resp.) if so is $\mathcal{X}$. The following statement immediately follows from Theorem~\ref{main}.

\begin{lemm}\label{uniformsring}
A Higmanian $S$-ring with parameters~$(f,m,n,k,t)$ is uniform if and only if Eq.~\eqref{uniform} holds.
\end{lemm}

\subsection{Relative difference sets and linked systems}

In this paper, we deal with closed linked systems of relative difference sets (RDSs) in the sense of~\cite{MR}. Let us recall the definitions. Let $G$ be a finite group, $N\leq G$, and $\mathcal{L}=\{X_{\alpha}:~\alpha\in W\}$ a collection of RDSs in $G$ with the same forbidden subgroup~$N$ and parameters~$(m,n,k,\lambda)$ indexed by elements of a finite set $W$ with $|W|=w\geq 2$. The collection $\mathcal{L}$ is called a \emph{closed linked system} of RDSs relative to $N$ if there exist a bijection $\chi=\chi_{\mathcal{L}}:W \rightarrow W$, a function $\psi=\psi_{\mathcal{L}}:(W\times W)\setminus \{(\alpha,\chi(\alpha))\in W\times W:~\alpha\in W\}\rightarrow W$ and nonnegative integers $\mu$ and $\nu$ such that $X_{\alpha}^{(-1)}=X_{\chi(\alpha)}$ and
\begin{equation}\label{linked}
\underline{X_\alpha}\cdot\underline{X_\beta}=
\begin{cases}
ke+\lambda(\underline{G\setminus N}),~\beta=\chi(\alpha),\\
\mu \underline{X_{\psi(\alpha,\beta)}}+\nu (\underline{G\setminus X_{\psi(\alpha,\beta)}}),~\beta\neq \chi(\alpha).
\end{cases}
\end{equation}
for all $\alpha,\beta\in W$. We say that $(m,n,k,\lambda,w,\mu,\nu)$ are \emph{parameters} and $(\chi,\psi)$ is a \emph{pair of characteristic functions} of $\mathcal{L}$. 

Recall that an RDS is called \emph{semiregular} if $k=m$. If $\mathcal{L}$ is a closed linked system of semiregular RDSs, then $\mathcal{L}$ has parameters $(n\lambda,n,n\lambda,n,w,\mu,\nu)$, where 
\begin{equation}\label{munu}
\mu=\frac{1}{n}\left(n\lambda\pm (n-1)\sqrt{n\lambda}\right)~\text{and}~\nu=\frac{1}{n}\left(n\lambda\mp \sqrt{n\lambda}\right).
\end{equation}

If $\mathcal{L}$ is a closed linked system of semiregular RDSs indexed by the elements of the set $W$, then one can define the \emph{associate group} on the set $W^\infty=W\cup \{\infty\}$ (see~\cite[Section~4.2]{MR}), where $\infty $ is the identity element and the inversion operation $\widehat{\chi}$ and binary operation $\widehat{\psi}$ are induced by $\chi$ and $\psi$, respectively, as follows: 
$$\widehat{\chi}(\alpha)=
\begin{cases}
\chi(\alpha),~\alpha\neq \infty,\\
\infty,~\alpha=\infty,\\
\end{cases}$$ 
$$\widehat{\psi}(\alpha,\beta)=
\begin{cases}
\psi(\alpha,\beta),~\beta\neq \chi(\alpha),~\alpha,\beta\neq \infty,\\
\infty,~\beta=\widehat{\chi}(\alpha),\\
\alpha,~\alpha\neq\infty,~\beta=\infty,\\
\beta,~\alpha=\infty,~\beta\neq \infty.
\end{cases}$$

Given a prime power $q$ and a positive integer~$r$, the Heisenberg group over $\mathbb{F}_q$ of dimension~$2r+1$ and the elementary abelian group of order $q^{2r+1}$ are denoted by $H_{2r+1}(q)$ and $E(q^{2r+1})$, respectively. A central product of $r$ quaternion groups of order~$8$ is denoted by $Q_8^{[r]}$. The next lemma collects information from Lemma~5.4, Corollary~5.5, Corollary~6.6, Lemma~6.7, and Proposition~7.7 from~\cite{MR}.

\begin{lemm}\label{linkedexamples}
Let $q=p^i$, $i\geq 1$, be a prime power and $r$ a positive integer. The group $G$ from the first column of Table~$1$ has a closed linked system of semiregular RDSs with parameters and associated group from the second and third columns of Table~$1$, respectively. 
\end{lemm}

\begin{table}[h]

{\small
\begin{tabular}{|l|l|l|}
  \hline
  group $G$ & parameters &  associated group  \\
  \hline
	$Q_8^{[r]}$ & $(2^{2r},2,2^{2r},2^{2r-1},2,2^{2r-1}-2^r+2^{r-1},2^{2r-1}+2^{r-1})$  & $C_{3}$\\ \hline
  $H_{2r+1}(q)$, $q$ is odd & $(q^{2r},q,q^{2r},q^{2r-1},q,q^{2r-1}-q^r+q^{r-1},q^{2r-1}+q^{r-1})$  & $C_{q+1}$\\ \hline
  $E(q^{2r+1})$ & $(q^{2r},q,q^{2r},q^{2r-1},p^j-1,q^{2r-1}+q^r-q^{r-1},q^{2r-1}-q^{r-1})$, $j\leq i$ & $E(p^j)$\\  \hline
\end{tabular}
}
\caption{Closed linked systems of RDSs}
\end{table}

\section{Proof of Theorem~\ref{main}}

In this section, we prove Theorem~\ref{main}. Due to Lemma~\ref{criterion}, it suffices to prove the following statement.

\begin{theo}\label{main2}
A Higmanian scheme with parameters~$(f,m,n,k,t)$ is Q-Higmanian if and only if Eq.~\eqref{uniform} holds.
\end{theo}

\begin{proof}
Let $\mathcal{X}$ be a Higmanian scheme with parameters~$(f,m,n,k,t)$, $E$ and $F$ proper nontrivial parabolics of $\mathcal{X}$ such that $E\subsetneq F$, and $S$ and $T$ basic relations of $\mathcal{X}$ outside $F$ such that $n_T\leq n_S$. Then $f=\frac{|\Omega|}{n_F}$, $m=\frac{n_F}{n_E}$, $n=n_E$, $k=|\alpha S\cap \Delta|,~\Delta\in \Omega/F$, and $t=p_{TS}^{T}$ are the parameters of $\mathcal{X}$.

Let us order the basic relations of $\mathcal{X}$ in the following way:
$$R_0=\diag(\Omega\times \Omega),~R_1=E\setminus R_0,~R_2=S,~R_3=T,~R_4=F\setminus E.$$

\begin{lemm}\label{eigenmatrix}
There exists an ordering of primitive idempotents $E_0,\ldots,E_4$ of $\mathcal{X}$ such that the first eigenmatrix of $\mathcal{X}$ with respect to this ordering and the above ordering of basic relations can be written as
$$
P={\footnotesize
\begin{pmatrix}
1&n-1&k(f-1)&(mn-k)(f-1)&mn-n \\
1&-1&x_1&-x_1&0 \\
1&n-1&0&0&-n \\
1&-1&x_3&-x_3&0 \\
1&n-1&-k&-mn+k&mn-n
\end{pmatrix}},
$$
where $x_1$ and $x_3$ are the roots of the equation
\begin{equation}\label{square}
x^2+((f-2)(mn-k)-\frac{tmn}{k})x-\frac{(f-1)k(mn-k)}{m(n-1)}=0
\end{equation}
and $|x_1|\geq |x_3|$.
\end{lemm}

\begin{proof}
Expressing the parameters of a Higmanian scheme from~\cite[Section~10,~p.~217]{Hig} via $m$, $n$, $k$, and $f$, replacing $x_1$ and $x_2$ from~\cite[Section~10,~p.~217]{Hig} by $x_1$ and $x_3$, respectively, if $|x_1|\geq |x_3|$ and by $x_3$ and $x_1$, respectively, otherwise, and reordering the primitive idempotents according to the permutation $(1243)\in \sym(\{0,\ldots,4\})$, we obtain the required.
\end{proof}

Observe that $x_1$ and $x_3$ are real because $\mathcal{X}$ is symmetric. Further throughout the proof, we assume that the basic relations and primitive idempotents of $\mathcal{X}$ are ordered as in Lemma~\ref{eigenmatrix}. The next lemma immediately follows from Eq.~\eqref{mult} and Lemma~\ref{eigenmatrix}.

\begin{lemm}\label{multiplicities}
In the above ordering, the multiplicities of $\mathcal{X}$ are
$$m_0=1,~m_2=f(m-1),~m_4=f-1,$$
$$m_1=\frac{f(f-1)m(n-1)k(mn-k)}{(f-1)k(mn-k)+x_1^2m(n-1)},~m_3=\frac{f(f-1)m(n-1)k(mn-k)}{(f-1)k(mn-k)+x_3^2m(n-1)}.$$
\end{lemm}

Observe that in the above ordering, the adjacency algebra $\mathcal{M}=\mathcal{M}(\mathcal{X})$ of $\mathcal{X}$ has exactly two proper nontrivial $\circ$-subalgebras, namely $\langle E_0,E_4 \rangle$ and $\langle E_0,E_2,E_4\rangle$ corresponding to parabolics $F$ and $E$, respectively. Put $I=\{0,4\}$. Using Eq.~\eqref{krein}, Lemma~\ref{eigenmatrix}, and Lemma~\ref{multiplicities}, it is easy to check that 
$$q_{ij}^0=0~\text{and}~q_{ij}^4=0$$
whenever $i\neq j$ and $2\in\{i,j\}$ or $\{i,j\}\cap \{0,4\}\neq \varnothing$ and $\{i,j\}\cap \{1,3\}\neq \varnothing$. Therefore the sets
$$\{0,4\},~\{1,3\},~\{2\}$$
are the unions of some classes of the equivalence relation~$\sim_I$. On the other hand, one can verify using Eq.~\eqref{krein}, Lemma~\ref{eigenmatrix}, and Lemma~\ref{multiplicities} that $q_{04}^4\neq 0$ and
$$q_{13}^4=0~\text{if and only if}~x_1x_3=\frac{(f-1)^2k(mn-k)}{m(n-1)}.$$
However, the latter equality contradicts to Lemma~\ref{eigenmatrix} because $x_1$ and $x_3$ are roots of Eq.~\eqref{square} and hence $x_1x_3=-\frac{(f-1)k(mn-k)}{m(n-1)}$. Therefore $q_{13}^4\neq 0$ and consequently the sets $\{0,4\}$ and $\{1,3\}$ are exactly equivalence classes of $\sim_I$. Thus, the first two conditions from the definition of a Q-Higmanian scheme hold for the ordering of the primitive idempotents from Lemma~\ref{eigenmatrix}, $I=\{0,4\}$, and $l=2$.

Suppose that $\mathcal{X}$ is Q-Higmanian. The ordering from Lemma~\ref{eigenmatrix} is the unique ordering satisfying the following conditions: (1) $\langle E_0,E_4 \rangle$ is a $\circ$-subalgebra of $\mathcal{M}$; (2) $m_4\geq m_0$, $m_3\geq m_1$ (the latter holds because $|x_1|\geq |x_3|$). So $\mathcal{X}$ can be Q-Higmanian only for this ordering and $I=\{0,4\}$. Since $\rk(\mathcal{X})=5$, the number $l$ from the definition of a Q-Higmanian scheme should be equal to~$1$ or~$2$. However, if $l=1$, then each of the sets $\{1\}$, $\{2\}$, and $\{3\}$ should be an equivalence class of $\sim_I$ which is not true because $q_{04}^4\neq 0$ and $q_{13}^4\neq 0$. Therefore $l=2$. Together with the previous paragraph, this implies that $\mathcal{X}$ is Q-Higmanian if and only if Condition~$(3)$ from the definition of a Q-Higmanian scheme holds for $m_i$'s, $f$, and $l=2$, i.e. 
$$\frac{m_4}{m_0}=\frac{m_3}{m_1}=f-1.$$
Thus, to prove Theorem~\ref{main2} it suffices to verify that the latter equalities hold if and only if Eq.~\eqref{uniform} holds.

The equality $\frac{m_4}{m_0}=f-1$ follows immediately from Lemma~\ref{multiplicities}. Using the expressions for $m_1$ and $m_3$ from Lemma~\ref{multiplicities}, we conclude that the equality $\frac{m_3}{m_1}=f-1$ is equivalent to 
\begin{equation}\label{l1}
x_1^2-(f-1)x_3^2=\frac{(f-1)(f-2)k(mn-k)}{m(n-1)}.
\end{equation}
On the other hand, $x_1$ and $x_3$ are the roots of Eq.~\eqref{square} and hence 
$$x_1x_3=-\frac{(f-1)k(mn-k)}{m(n-1)}.$$
As $f\geq 2$ and $k<mn$, we have $x_1,x_3\neq 0$. Expressing $x_3$ via $x_1$ and substituting this expression into Eq.~\eqref{l1}, we obtain the biquadratic equation
$$x_1^4-a(f-2)x_1^2-(f-1)a^2=0,$$
where $a=\frac{(f-1)k(mn-k)}{m(n-1)}$. Since $a>0$, the above equation has two real solutions, namely $\pm\sqrt{(f-1)a}$, and hence
\begin{equation}\label{l2}
(x_1,x_3)=\left(\sqrt{(f-1)a},-\sqrt{\frac{a}{f-1}}\right)~\text{or}~(x_1,x_3)=\left(-\sqrt{(f-1)a},\sqrt{\frac{a}{f-1}}\right).
\end{equation}
Observe that
$$x_1+x_3=-(f-2)(mn-k)+\frac{tmn}{k}$$
because $x_1$ and $x_3$ are the roots of Eq.~\eqref{square}. So
$$t=\frac{k}{mn}(x_1+x_3+(f-2)(mn-k)).$$
Together with Eq.~\eqref{l2}, the above equality implies Eq.~\eqref{uniform}. On the other hand, if Eq.~\eqref{uniform} holds, then substituting the expression for $t$ into Eq.~\eqref{square}, one can find $x_1$ and $x_3$ and verify that Eq.~\eqref{l2} holds which yields that Eq.~\eqref{l1} holds as desired.
\end{proof}

\section{Examples of uniform Higmanian Cayley schemes}

In this section, we provide examples of uniform Higmanian Cayley schemes. To do this, we construct uniform Higmanian $S$-rings whose corresponding Cayley schemes are uniform and Higmanian. 

\subsection{Example~1} 
Let $G$ be a finite abelian group and $X$ a divisible difference set in $G$ relative to a subgroup $N$ with parameters $(m,n,k,\lambda_1,\lambda_2)$. Suppose that $X$ satisfies the \emph{intersection condition}, i.e. the number $|X\cap Ng|$ does not depend on $g\in G$ or, equivalently, $X$ contains the same number of elements from each right $N$-coset. Then the partition of the generalized dihedral group $\langle G,u \rangle$, where $|u|=2$ and $u$ inverses every element of $G$, into the sets
$$T_0=\{e\},~T_1=N^\#,~T_2=G\setminus N,~T_3=Xu,~T_4=(G\setminus X)u$$
defines a Higmanian $S$-ring $\mathcal{A}$ over $\langle G,u \rangle$ by~\cite[Lemma~6.2]{Ry2}. The parameter $f$ of $\mathcal{A}$ is equal to $|\langle G,u \rangle:G|=2$. So the right-hand side of Eq.~\eqref{uniform} is equal to~$0$. One can see that $\underline{T_3}\cdot\underline{T_4}\in \mathbb{Z}G$ and hence $p_{T_3T_4}^{T_3}=p_{T_4T_3}^{T_4}=0$. This implies that the parameter $t$ of $\mathcal{A}$ is equal to~$0$. Therefore the left-hand side of Eq.~\eqref{uniform} is also equal to~$0$ and consequently Eq.~\eqref{uniform} is attained. Thus, $\mathcal{A}$ and the Higmanian scheme $\mathcal{X}(\mathcal{A})$ are uniform by Theorem~\ref{main}. The parameters of $\mathcal{A}$ and $\mathcal{X}(\mathcal{A})$ are $(2,m,n,k,0)$.

\subsection{Example~2} 

The second example is based on closed linked systems of relative difference sets. Let $G$ be a finite group, $N$ a subgroup of $G$, and $\mathcal{L}=\{X_\alpha:~\alpha\in W\}$ a closed linked system of semiregular RDSs in $G$ with forbidden subgroup $N$, parameters~$(n\lambda,n,n\lambda,\lambda,w,\mu,\nu)$, and a pair of characteristic functions~$(\chi,\psi)$. Suppose that $U$ is a group isomorphic to the group $W^\infty$ associated with $\mathcal{L}$, the operations $\widehat{\chi}$ and $\widehat{\psi}$ are defined as in Section~2.5, and $\varphi$ is an isomorphism from $U$ to $W^\infty$. By the definition, $|U|=w+1$. Since $\infty$ is an identity element of $W^\infty$, we have $\varphi(e)=\infty$ and hence $\varphi(U^\#)=W$. A partition $\mathcal{S}=\mathcal{S}(\varphi)$ of $G\times U$ is defined as follows:
$$T_0=\{e\},~T_1=N^\#,~T_2=G\setminus N,$$
$$T_3=\bigcup \limits_{u\in U^\#} uX_{\varphi(u)},~T_4=\bigcup \limits_{u\in U^\#} u(G\setminus X_{\varphi(u)}).$$ 
Put $\mathcal{A}=\mathcal{A}(\varphi)=\Span_{\mathbb{Z}}\{\underline{T}:~T\in \mathcal{S}\}$.

It is possible to check that the adjacency matrices of the binary relations $R(T_i)=\{(g,tg)\in (G\times U)^2:~g\in G\times U,~t\in T_i\}$, $i\in \{0,\ldots,4\}$, are of the forms from~\cite[Eqs.~(6)-(8)]{KhS} and the linked system of divisible designs corresponding to $\mathcal{L}$ satisfies the conditions of~\cite[Theorem~4.1]{KhS} which implies, in fact, that $\mathcal{A}$ is a Higmanian $S$-ring. However, in the context of this paper, it seems shorter and more natural to provide an explicit computation in the group ring $\mathbb{Z}G$ which yields the latter. 

\begin{prop}\label{higman}
The $\mathbb{Z}$-module $\mathcal{A}$ is a Higmanian $S$-ring over $G\times U$ with parameters 
$$(w+1,n\lambda,n,n\lambda(n-1),(n-1)(w-1)\nu).$$
\end{prop}

\begin{proof}
At first, let us check that $T^{(-1)}=T$ for each $T\in \mathcal{S}$. This is obvious for $T\in\{T_0,T_1,T_2\}$. One can see that given $u\in U$, 
$$(uX_{\varphi(u)})^{(-1)}=u^{-1}(X_{\varphi(u)})^{(-1)}=u^{-1}X_{\widehat{\chi}(\varphi(u))}=u^{-1}X_{\varphi(u^{-1})},$$
where the second equality holds by the definition of $\widehat{\chi}$, whereas the third one holds because $\varphi$ is an isomorphism from $U$ to $W^\infty$. The above equality implies that $T_3^{(-1)}=T_3$. So 
$$T_4^{(-1)}=((G\times U)\setminus (T_0\cup T_1 \cup T_2 \cup T_3))^{(-1)}=(G\times U)\setminus (T_0\cup T_1 \cup T_2 \cup T_3)=T_4$$
and we are done. 

Now let us verify that $\underline{T_i}\cdot\underline{T_j}\in \mathcal{A}$ for all $i,j\in\{0,1,2,3,4\}$. Since $\underline{T_4}=\underline{G\times U}-\sum \limits_{i=0}^{3} \underline{T_i}$, the above inclusion can be verified only for $i,j\in\{0,1,2,3\}$. If $i=0$ or $j=0$, then this is obvious. It is easy to check using Eq.~\eqref{easy} that
$$\underline{T_1}\cdot\underline{T_1}=(n-1)\underline{T_0}+(n-2)\underline{T_1},~\underline{T_1}\cdot\underline{T_2}=\underline{T_2}\cdot\underline{T_1}=(n-1)\underline{T_2},$$
$$\underline{T_2}\cdot\underline{T_2}=(n^2\lambda-n)(\underline{T_0}+\underline{T_1})+(n^2\lambda-2n)\underline{T_2}.$$
Since each $X_{\varphi(u)}$ is a transversal for $N$ in $G$, we have 
$$\underline{X}_{\varphi(u)}\cdot \underline{T_1}=\underline{T_1}\cdot\underline{X}_{\varphi(u)}=\underline{X}_{\varphi(u)}(\underline{N}-e)=\underline{G}-\underline{X}_{\varphi(u)}.$$ 
Therefore
$$\underline{T_3}\cdot\underline{T_1}=\underline{T_1}\cdot\underline{T_3}=\left(\sum\limits_{u\in U^\#} u\underline{X}_{\varphi(u)}\right)\cdot\underline{T_1}=\sum\limits_{u\in U^\#} u(\underline{G}-\underline{X}_{\varphi(u)})=\underline{T_4}$$
and
$$\underline{T_3}\cdot\underline{T_2}=\underline{T_2}\cdot\underline{T_3}=\left(\sum\limits_{u\in U^\#} u\underline{X}_{\varphi(u)}\right)\cdot(\underline{G}-\underline{N})=(n\lambda-1)\underline{U^\#G}=(n\lambda-1)(\underline{T_3}+\underline{T_4}).$$
Finally,
$$\underline{T_3}\cdot\underline{T_3}=\sum \limits_{u\in U^\#} \underline{X}_{\varphi(u)}\cdot \underline{X}_{\varphi(u^{-1})}+\sum \limits_{u\in U^\#} u\sum \limits_{v\in U^\#\setminus \{u\}} \underline{X}_{\varphi(v)}\cdot \underline{X}_{\varphi(v^{-1}u)}=$$
$$=\sum \limits_{u\in U^\#}\underline{X}_{\varphi(u)}\cdot \underline{X}_{\widehat{\chi}(\varphi(u))}+\sum \limits_{u\in U^\#} u\sum \limits_{v\in U^\#\setminus \{u\}} \underline{X}_{\varphi(v)}\cdot \underline{X}_{\widehat{\psi}(\widehat{\chi}(\varphi(v)),\varphi(u))}=$$
$$=\sum \limits_{u\in U^\#}\left(n\lambda e+\lambda(\underline{G}-\underline{N})\right)+\sum \limits_{u\in U^\#} u \sum \limits_{v\in U^\#\setminus \{u\}} \left(\mu \underline{X_{\varphi(u)}}+\nu\left(\underline{G}-\underline{X_{\varphi(u)}}\right)\right)=$$
$$=w\left(n\lambda e+\lambda(\underline{G}-\underline{N})\right)+(w-1)\sum \limits_{u\in U^\#} \left(\mu u\underline{X}_{\varphi(u)}+\nu u\left(\underline{G}- \underline{X}_{\varphi(u)}\right)\right)=$$
$$=wn\lambda\underline{T_0}+w\lambda\underline{T_2}+(w-1)\mu\underline{T_3}+(w-1)\nu \underline{T_4},$$
where the second equality holds because $\varphi$ is an isomorphism from $U$ to $W^\infty$, whereas the third equality holds because $\widehat{\psi}$ is associative and $\widehat{\chi}(\varphi(v))$ is the element inverse to~$\varphi(v)$ in the group $W^\infty$.

Thus, $\mathcal{A}$ is an $S$-ring over $G\times U$. Clearly, $\rk(\mathcal{A})=5$. One can see that $\rk(\mathcal{A}_N)=\rk(\mathcal{A}_{(G\times U)/G})=2$ and $\rk(\mathcal{A}_{(G\times U)/N})=\rk(\mathcal{A}_G)=3$. Therefore, $\mathcal{A}$ is Higmanian. Since every basic set of $\mathcal{A}$ is inverse-closed, $\mathcal{A}$ is symmetric and hence commutative.

One can see that $|T_3|=wn\lambda$ whereas $|T_4|=wn\lambda(n-1)$. So $|T_3|\leq |T_4|$. The first four parameters of $\mathcal{A}$ are equal to~$(w+1,n\lambda,n,n\lambda(n-1))$ by the definition of $\mathcal{A}$. The last parameter of $\mathcal{A}$ is $t=p_{T_3T_4}^{T_3}$. The above computation of $\underline{T_3}\cdot\underline{T_3}$ implies that $p_{T_3T_3}^{T_4}=(w-1)\nu$. Using Eq.~\eqref{triangle0} and symmetry of $\mathcal{A}$, one can compute that
$$t=p_{T_3T_4}^{T_3}=\frac{|T_4|}{|T_3|}p_{T_3T_3}^{T_4}=(n-1)(w-1)\nu$$
as required.
\end{proof}

In fact, the Higmanian $S$-ring $\mathcal{A}(\varphi)$ does not depend on the choice of~$\varphi$. Further we use the notation $T_i(\varphi)$ for the corresponding basic set of the $S$-ring $\mathcal{A}(\varphi)$. 

\begin{lemm}\label{cayleyisom}
Let $\varphi_1$ and $\varphi_2$ be isomorphisms from $U$ to $W^\infty$. Then the $S$-rings $\mathcal{A}(\varphi_1)$ and $\mathcal{A}(\varphi_2)$ are Cayley isomorphic.
\end{lemm}

\begin{proof}
Let $\varphi(u)=\varphi_2^{-1}(\varphi_1(u))\in \aut(U)$ and $\widehat{\varphi}\in \aut(G\times U)$ be such that $\widehat{\varphi}^G=\id_G$ and $\widehat{\varphi}^U=\varphi$. Let us prove that $\widehat{\varphi}$ is a Cayley isomorphism from $\mathcal{A}(\varphi_1)$ to $\mathcal{A}(\varphi_2)$. It is easy to see that $\widehat{\varphi}(N^\#)=N^\#$ and $\widehat{\varphi}(G\setminus N)=G\setminus N$ because $\widehat{\varphi}^G=\id_G$. Further,
$$\widehat{\varphi}(T_3(\varphi_1))=\widehat{\varphi}(\bigcup \limits_{u\in U^\#} uX_{\varphi_1(u)})=\bigcup \limits_{u\in U^\#} \varphi(u)X_{\varphi_1(u)}=\bigcup \limits_{u\in U^\#} \varphi(u)X_{\varphi_2(\varphi(u))}=\bigcup \limits_{v\in U^\#} vX_{\varphi_2(v)}=T_3(\varphi_2),$$
where $v=\varphi(u)$ runs over $U^\#$ because $\varphi\in \aut(U)$. Finally,
$$\widehat{\varphi}(T_4(\varphi_1))=\widehat{\varphi}((G\times U)\setminus (G\cup T_3(\varphi_1)))=(G\times U)\setminus (G\cup T_3(\varphi_2))=T_4(\varphi_2).$$
Thus, $\widehat{\varphi}$ maps the basic sets of $\mathcal{A}(\varphi_1)$ to the basic sets of $\mathcal{A}(\varphi_2)$, i.e. $\widehat{\varphi}$ is a Cayley isomorphism from $\mathcal{A}(\varphi_1)$ to $\mathcal{A}(\varphi_2)$.
\end{proof}

Let $\mathcal{X}=\mathcal{X}(\mathcal{A})$ be the Cayley scheme corresponding to $\mathcal{A}$. Clearly, $\mathcal{X}$ is Higmanian. It follows almost immediately that it is also uniform. 

\begin{corl}\label{uniformlinked}
The $S$-ring $\mathcal{A}$ and the Cayley scheme $\mathcal{X}$ are uniform.
\end{corl}

\begin{proof}
An explicit substitution of parameters~$(w+1,n\lambda,n,n\lambda(n-1),(n-1)(w-1)\nu)$ into Eq.~\eqref{uniform} leads to the equality
$$\nu=\frac{n\lambda\mp \sqrt{n\lambda}}{n}$$
which is the second equality from Eq.~\eqref{munu}. Therefore $\mathcal{A}$ and hence $\mathcal{X}$ are uniform by Theorem~\ref{main}.
\end{proof}

\begin{rem}
As was mentioned in the introduction, a connection between uniform and Q-polynomial schemes was studied in the paper~\cite{DMM}. In particular, in that paper it was proved that every uniform indecomposable scheme of rank~$4$ is Q-polynomial~\cite[Proposition~6.1]{DMM}. However, the latter does not hold for Higmanian schemes. One can verify using a criterion for Higmanian schemes to be Q-polynomial~\cite [Proposition~7.2]{DMM} that the uniform schemes from Proposition~\ref{higman} are not Q-polynomial. 
\end{rem}

The next corollary immediately follows from Lemma~\ref{linkedexamples}, Proposition~\ref{higman}, and Corollary~\ref{uniformlinked}.

\begin{corl}\label{higmangroup}
Let $q=p^i$ be a prime power and $r$ a positive integer. There is a uniform Higmanian Cayley scheme over a group $G$ from the first column of Table~$2$ with parameters from the second column of Table~$2$, respectively. 
\end{corl}

\begin{table}[h]

{\small
\begin{tabular}{|l|l|}
  \hline
  group $G$ & parameters   \\
  \hline
	$Q_8^{[r]}$ & $(3,2^{2r},2,2^{2r},2^{r-1}(2^r+1))$ \\ \hline
  $H_{2r+1}(q)\times C_{q+1}$, $q$ is odd & $(q+1,q^{2r},q,q^{2r}(q-1),q^{r-1}(q-1)^2(q^r+1))$ \\ \hline
  $E(q^{2r+1})\times E(p^j)$, $j\leq i$ & $(p^j,q^{2r},q,q^{2r}(q-1),q^{r-1}(q-1)(q^r-1)(p^j-2))$ \\  \hline
\end{tabular}
}
\caption{Parameters of uniform Higmanian Cayley schemes.}
\end{table}

\end{document}